\documentclass{article}
\usepackage{amsmath,amssymb,amsthm,graphicx,amsxtra}
\usepackage[hypertex]{hyperref}

\usepackage[all]{xy}
\newtheorem{thm}{Theorem}[section]
\newtheorem{prop}[thm]{Proposition}
\newtheorem{lem}[thm]{Lemma}
\newtheorem{cor}[thm]{Corollary}

\newtheorem{conj}[thm]{Conjecture}
\newtheorem*{mthm}{Main Theorem}

\theoremstyle{definition}
\newtheorem{dfn}[thm]{Definition}
\newtheorem{ass}[thm]{Assumption}
\newtheorem{ex}[thm]{Example}
\newtheorem*{notation}{Notation}

\begin{document}
\title{A note on logarithmic growth of solutions of\\ $p$-adic differential equations without solvability}
\author{Shun Ohkubo
\footnote{
Graduate School of Mathematics, Nagoya University, Furocho, Chikusaku, Nagoya 464-8602, Japan. E-mail address: shun.ohkubo@gmail.com}
}
\maketitle

\begin{abstract}
For a $p$-adic differential equation solvable in an open disc (in a $p$-adic sense), around 1970, Dwork proves that the solutions satisfy a certain growth condition on the boundary. Dwork also conjectures that a similar phenomenon should be observed without assuming the solvability. In this paper, we verify Dwork's conjecture in the rank two case, which is the first non-trivial result on the conjecture. The proof is an application of Kedlaya's decomposition theorem of $p$-adic differential equations defined over annulus.
\end{abstract}

\tableofcontents
\section{Introduction}
Cauchy's theorem on ordinary linear differential equations over $\mathbb{C}$ asserts that a differential equation
\[
\frac{d^mf}{dz^m}+a_{m-1}(z)\frac{d^{m-1}f}{dz^{m-1}}+\dots+a_0(z)f=0
\]
where $a_i(z)\in\mathbb{C}[\![z]\!]$ converging on the open unit disc $D$ has a full set of solutions on $D$ (\cite[Theorem~7.2.1]{pde}). In the $p$-adic setting, a na\"ive analogue of Cauchy's theorem fails due to the absence of ``$p$-adic'' analytic continuation. Even the exponential series $\exp{(t)}$ converges only on $|t|<p^{-1/(p-1)}$.

After his proof of the rationality part of Weil conjecture using $p$-adic differential equations, Dwork starts to systematically study $p$-adic differential equations in the 1960s. Thanks to efforts of Dwork and his successors, we can measure the obstruction that local solutions $f$ of $p$-adic differential equations to global solutions firstly by the radius of convergence of $f$, secondly by growth of $f$ on the boundary of its convergence disc.

This paper mainly concerns the latter invariant, so called logarithmic growth (log-growth for short) of solutions of $p$-adic differential equations. More concretely, we prove the first non-trivial instance of the following conjecture of Dwork.

Let us fix notation. Let $K$ be a complete discrete valuation field of mixed characteristic $(0,p)$. Let $K\{t\}$ be the ring of power series over $K$ converging on the $p$-adic open unit disc $|t|<1$. Let $K[\![t]\!]_0$ be the ring consisting of $f\in K\{t\}$ which is bounded on $|t|<1$. Let $R$ denote either $K[\![t]\!]_0$ or $K\{t\}$. Then $R$ is endowed with the derivation $d=d/dt$. A {\it differential module} over $R$ is a finite free $R$-module $M$ endowed with a differential $D$ compatible with $d$, that is, an additive map $D:M\to M$ satisfying $D(r\cdot m)=dr\cdot m+r\cdot D(m)$ for all $r\in R,m\in M$. We put $M^{\nabla=0}=\ker{D}$.

A power series $\sum_{i\in\mathbb{N}}{a_it^i}$ with $a_i\in K$ {\it has order of log-growth} at most $\delta\in\mathbb{R}_{\ge 0}$ if $\sup_{i\in\mathbb{N}}{\{|a_i|/(i+1)^{\delta}\}}<\infty$. We denote by $K[\![t]\!]_{\delta}$ the set of power series over $K$ having order of log-growth at most $\delta$. For $\delta<0$, we set $K[\![t]\!]_{\delta}=K[\![t]\!]_0$ for simplicity. Recall that $K[\![t]\!]_{\delta}$ is a $K[\![t]\!]_0$-submodule of $K\{t\}$, and is stable under $d$ (\cite[Proposition 1.2 (1), (3), (4)]{CT}). For a differential module $M$ over $K[\![t]\!]_0$, we regard $M\otimes_{K[\![t]\!]_0}K\{t\}$ as a differential module over $K\{t\}$ by extending $D$ by the formula $D(m\otimes f)=D(m)\otimes f+m\otimes df$ for $m\in M,f\in K\{t\}$. Then $M\otimes_{K[\![t]\!]_0}K[\![t]\!]_{\delta}$ is stable under $D$. We denote by $(M\otimes_{K[\![t]\!]_0}K[\![t]\!]_{\delta})^{\nabla=0}$ the kernel of $D$ restricted to $M\otimes_{K[\![t]\!]_0}K[\![t]\!]_{\delta}$.

We would like to generalize the following fundamental theorem on this topic:
\begin{thm}[{\cite[Theorem~1]{Dw}}]\label{thm:Dw}
Let $M$ be a differential module over $K[\![t]\!]_0$ of rank $m$. Assume that
\[
m=\dim_K(M\otimes_{K[\![t]\!]_0}K\{t\})^{\nabla=0}.
\]
Then,
\[
(M\otimes_{K[\![t]\!]_0}K[\![t]\!]_{m-1})^{\nabla=0}=(M\otimes_{K[\![t]\!]_0}K\{t\})^{\nabla=0}.
\]
\end{thm}

The assumption of the theorem, which can be regarded as a solvability condition, is satisfied, for example, when $M$ admits a Frobenius structure, or, $M$ comes from geometry via Gauss-Manin connection. See \cite{CT}, \cite{CT2}, and \cite{Ohk} in recent developments in the setup assuming Frobenius structures.

When the solvability fails, Dwork proposes

\begin{conj}[{\cite[Conjecture~1]{Dw}}]\label{conj:Dw}
Let $M$ be a differential module over $K[\![t]\!]_0$. Let
\[
m'=\dim_K(M\otimes_{K[\![t]\!]_0}K\{t\})^{\nabla=0}.
\]
Then,
\[
(M\otimes_{K[\![t]\!]_0}K[\![t]\!]_{m'-1})^{\nabla=0}=(M\otimes_{K[\![t]\!]_0}K\{t\})^{\nabla=0}.
\]
\end{conj}

The case where $m'$ is equal to the rank of $M$ is nothing but Theorem~\ref{thm:Dw}. In the case $m'=0$, we have nothing to prove. In particular, the conjecture is trivial modulo Theorem~\ref{thm:Dw} in the rank one case. Besides these cases, as long as the author knows, no general result on the conjecture is known. Our main result verifies the conjecture in the first non-trivial case:

\begin{mthm}\label{thm:mthm}
Conjecture~\ref{conj:Dw} is true if $\mathrm{rank}\ M=2$.
\end{mthm}

The main ingredient of the proof is Kedlaya's decomposition theorem of $p$-adic differential equations by the intrinsic generic subsidiary radii which is a refinement of the generic radius of convergence (Theorem~\ref{thm:dec}). Precisely speaking, the strong decomposition over $\mathcal{E}$ extends to an annulus $K\langle\alpha/t,t]\!]_0$, which reduces to studying some rank one objects (Corollary~\ref{cor:C}). In the higher rank case, our strategy does not seem to work since we do not know whether an analogue of Theorem~\ref{thm:Dw} over $K\langle\alpha/t,t]\!]_0$ holds.

In Appendix 1, we give an alternative proof of Theorem~\ref{thm:dec}. In Appendix 2, we give an explicit example of a differential module over $K[\![t]\!]_0$ of rank two, and explain Main Theorem and Theorem~\ref{thm:dec} by using this example.
\subsection*{Notation}

In this paper, we adopt the notation in \cite{pde}. Let $p$ be a prime number, $K$ a complete non-archimedean valuation field. Denote by $\mathcal{O}_K$ the integer ring of $K$. Let $|\ |:K\to\mathbb{R}_{\ge 0}$ be the multiplicative norm on $K$ normalized by $|p|=p^{-1}$. We define the valuation $v$ on $K$ by $v(\cdot)=-\log{|\cdot|}$. Put $\omega=p^{-1/(p-1)}=|p^{1/(p-1)}|<1$.

Let $(G,|\cdot|)$ be a normed abelian group, and $T:G\to G$ be an endomorphism of $G$. We define the operator norm and spectral norm of $T$ by
\begin{align*}
|T|_{\mathrm{op},G}&=\sup_{g\in G,g\neq 0}\{|T(g)|/|g|\},\\
|T|_{\mathrm{sp},G}&=\lim_{s\to\infty}|T^s|^{1/s}_{\mathrm{op},G}.
\end{align*}
\section{Local decomposition by subsidiary radii}\label{sec:loc}

We recall the decomposition of differential modules over complete non-archimedean valuation field proved by Kedlaya, based on works of Christol-Dwork and Robba.

\begin{dfn}[{\cite[Defintions~9.4.1, 9.4.3]{pde}}]\label{dfn:Amice}
Let $\rho\in (0,\infty)$. The $\rho$-{\it Gauss norm} $|\ |_{\rho}$ on $K(t)$ is the unique multiplicative norm satisfying $|\sum{a_it^i}|_{\rho}=\sup{|a_i|\rho^i}$ for $a_i\in K$. Let $F_{\rho}$ be the completion of $K(t)$ under $|\ |_{\rho}$.

We also denote by $|\ |_1$ the multiplicative norm on $\mathcal{O}_K[\![t]\!]$ defined by $|\sum{a_it^i}|_1=\sup{|a_i|}$ for $a_i\in\mathcal{O}_K$. We define $\mathcal{E}$ as the completion of the fraction field of $\mathcal{O}_K[\![t]\!]$ under $|\ |_1$.

Both $F_{\rho}$ and $\mathcal{E}$ are equipped with the derivation $d=d/dt$. 
\end{dfn}

\begin{ass}[{\cite[Definition~1.4.1]{KX}}]
Let $(\mathcal{K},|\cdot|)$ be a complete non-archimedean valuation field (may not be discretely valued) of mixed characteristic $(0,p)$. We assume that $(\mathcal{K},\partial)$ is a differential field of rational type, i.e.,
\[
\partial:\mathcal{K}\to\mathcal{K}
\]
is a derivation and there exists $u\in\mathcal{K}$ such that
\begin{enumerate}
\item[(a)] We have $\partial(u)=1$ and $|\partial|_{\mathrm{op},\mathcal{K}}=|u|^{-1}$,
\item[(b)] We have $|\partial|_{\mathrm{sp},\mathcal{K}}\le\omega |\partial|_{\mathrm{op},\mathcal{K}}$.
\end{enumerate}
\end{ass}

\begin{ex}\label{ex:field}
In this paper, we consider only the cases where
\[
(\mathcal{K},\partial)=(F_{\rho},d), (\mathcal{E},d).
\]
Both satisfy the assumptions~(a) and (b) with $u=t$ (\cite[Defnition~9.4.1]{pde}).
\end{ex}

\begin{dfn}[{\cite[Definitions~1.2.2, 1.2.6, 1.2.8]{KX}}]
A {\it differential module} over $\mathcal{K}$ is a finite dimensional $\mathcal{K}$-vector space $V$ equipped with an action of $\partial$. Define the {\it intrinsic generic radius of convergence} of $V$ as
\[
IR(V)=|\partial|_{\mathrm{sp},\mathcal{K}}/|\partial|_{\mathrm{sp},V}\in (0,1]
\]
for any fixed $\mathcal{K}$-compatible norm $|\cdot|$ on $V$. The intrinsic generic radius of convergence does not depend on the choice. Let $V_1,\dots,V_m$ be the Jordan-H\"older constituents of $V$ in the category of differential modules over $\mathcal{K}$. We define the {\it intrinsic generic subsidiary radii of convergence} $\mathfrak{IR}(V)$ as the multiset consisting of $IR(V_i)$ with multiplicity $\dim_{\mathcal{K}}V_i$ for $i=1,\dots,m$. Let $IR(V;1)\le\dots\le IR(V;\dim_{\mathcal{K}}V)$ denotes the elements in $\mathfrak{IR}(V)$ in increasing order.

Any differential module isomorphic to a direct sum of copies of $(\mathcal{K},d)$ is said to be {\it trivial}.

We say that $V$ has {\it pure radii} if $\mathfrak{IR}(V)$ consists of $\dim_{\mathcal{K}}V$ copies of $IR(V)$.
\end{dfn}

\begin{ex}\label{ex:rank one}
\begin{enumerate}
\item[(i)] (\cite[Example~9.5.2]{pde}) Let $\mathcal{K}$ be as in Example~\ref{ex:field}. For a fixed $\lambda\in K$, let $V=\mathcal{K}e$ be the rank one differential module over $\mathcal{K}$ defined by $\partial (e)=\lambda t^{-1}e$. Then, $IR(V)=1$ if and only if $\lambda\in\mathbb{Z}_p$.
\item[(ii)] We consider the differential modules $V=\mathcal{E}e$ over $\mathcal{E}$ defined by $\partial(e)=-te$. We will prove $IR(V)=\omega$. Since $|\partial|_{\mathrm{sp},\mathcal{E}}=\omega$ (\cite[Definition~9.4.1]{pde}), it suffices to prove $|\partial|_{\mathrm{sp},V}=1$. For $i,j\in\mathbb{N}$, by induction on $i$, we have $\partial^i(t^je)=((-1)^it^{i+j}+\varepsilon_{ij}(t))e$ for some polynomial $\varepsilon_{ij}(t)\in\mathbb{Z}+\mathbb{Z}t+\dots+\mathbb{Z}t^{i+j-1}$. Therefore, $|\partial|_{\mathrm{op},V}\le 1$ hence $|\partial|_{\mathrm{sp},V}\le 1$. Since $|\partial^i(e)|=|(-t)^i+\varepsilon_{i0}(t)|_1|e|=|e|$, we also have $|\partial^i|_{\mathrm{op},V}\ge 1$ hence $|\partial|_{\mathrm{sp},V}\ge 1$, which implies the assertion.
\end{enumerate}
\end{ex}

\begin{lem}[{\cite[Lemma~1.2.9]{KX}}]\label{lem:exact}
For an exact sequence of non-zero differential modules over $\mathcal{K}$
\[
0\to V'\to V\to V''\to 0,
\]
we have
\[
\mathfrak{IR}(V)=\mathfrak{IR}(V')\cup\mathfrak{IR}(V'').
\]
\end{lem}

\begin{thm}[{\cite[Theorem~1.4.21]{KX}}]\label{thm:local}
Let $V$ be a differential module over $\mathcal{K}$. Then, there exists a decomposition of differential modules over $\mathcal{K}$
\[
V=\displaystyle\bigoplus_{r\in (0,1]}V_r,
\]
where $V_r$ has pure radii $r$.
\end{thm}


\section{Global decomposition by subsidiary radii}

We recall one of ``globalizations'' of Theorem~\ref{thm:local}.

In the rest of this section, assume that $K$ is discretely valued.

\begin{dfn}[{\cite[\S\S~8.1, 15.1]{pde}}]\label{dfn:ring}
We have the ring of (bounded) analytic functions on the open unit disc $0\le |t|<1$
\begin{align*}
K[\![t]\!]_0&=\left\{\sum_{i\in\mathbb{N}}a_it^i\in K[\![t]\!];a_i\in K,\ \sup_{i\in\mathbb{N}}{|a_i|}<\infty\right\}=\mathcal{O}_K[\![t]\!]\otimes_{\mathcal{O}_K}K,\\
K\{t\}&=\left\{\sum_{i\in\mathbb{N}}a_it^i\in K[\![t]\!];a_i\in K, \lim_{i\to\infty}{|a_i|\eta^i}=0\ (\eta\in (0,1))\right\}.
\end{align*}
We endow $K[\![t]\!]_0$ (resp. $K\{t\}$) with (natural extensions of) Gauss norms $|\cdot|_{\alpha}$ for $\alpha\in [0,1]$ (resp. $\alpha\in [0,1)$).

We have the ring of (bounded) analytic functions on the annulus $\alpha\le |t|<1$ for $\alpha\in (0,1)$
\begin{align*}
K\langle\alpha/t,t]\!]_0&=\left\{\sum_{i\in\mathbb{Z}}a_it^i;a_i\in K, \lim_{i\to-\infty}{|a_i|\alpha^i}=0,\ \sup_{i\in\mathbb{Z}}{|a_i|}<\infty\right\},\\
K\langle\alpha/t,t\}&=\left\{\sum_{i\in\mathbb{Z}}a_it^i;a_i\in K, \lim_{i\to\pm\infty}{|a_i|\eta^i}=0\ (\eta\in (\alpha,1))\right\}.
\end{align*}
We endow $K\langle\alpha/t,t]\!]_0$ (resp. $K\langle\alpha/t,t\}$) with Gauss norms $|\cdot|_{\rho}$ for $\rho\in [\alpha,1]$ (resp. $\rho\in [\alpha,1)$).

We define the (bounded) Robba ring by
\[
\mathcal{E}^{\dagger}=\bigcup_{\alpha\in (0,1)}K\langle\alpha/t,t]\!]_0,\ \mathcal{R}=\bigcup_{\alpha\in (0,1)}K\langle\alpha/t,t\}.
\]
Each of these rings is equipped with the derivation $d=d/dt$.

For simplicity, we define $K\langle 0/t,t]\!]_0$ and $K\langle 0/t,t\}$ as $K[\![t]\!]_0$ and $K\{t\}$.
\end{dfn}

We explain relations between these rings and those defined in \S~\ref{sec:loc}. The ring $K[t][t^{-1}]$ is dense in $K\langle\alpha/t,t]\!]_0$ with respect to $|\ |_{\rho}$ for $\rho\in [\alpha,1)$, $F_{\rho}$ is obtained as the completion of the fraction field of $K\langle\alpha/t,t]\!]_0$ with respect to $|\ |_{\rho}$ for $\rho\in [\alpha,1)$. By a similar reason, for $\rho\in [\alpha,1)$, $F_{\rho}$ is obtained as the completion of the fraction field of $K\langle\alpha/t,t\}$ with respect to $|\ |_{\rho}$.

The ring $K[\![t]\!]_0[t^{-1}]$ is dense in $K\langle\alpha/t,t]\!]_0$ with respect to $|\ |_1$, $\mathcal{E}$ is obtained as the completion of the fraction field of $K\langle\alpha/t,t]\!]_0$ with respect to $|\ |_1$.

\begin{dfn}
Let $\delta\in [0,\infty)$. A power series $f\in K\{t\}$ has {\it order of log-growth} at most $\delta$ if
\[
|f|_{\rho}=O((-\log{(1/\rho)})^{\delta})\text{ as }\rho\uparrow 1.
\]
We denote by $K[\![t]\!]_{\delta}$ the set consisting of $f\in K\{t\}$ having order of log-growth at most $\delta$. Note that $K[\![t]\!]_{\delta}$ is an increasing sequence of $K[\![t]\!]_0$-modules and, for a power series $f=\sum_{i\in\mathbb{N}}{a_it^i}$ with $a_i\in K$, $f\in K[\![t]\!]_{\delta}$ if and only if $\sup_{i\in\mathbb{N}}\{|a_i|/(i+1)^{\delta}\}<\infty$ (\cite[Proposition~2.3.3]{Chr}).
\end{dfn}

\begin{dfn}[{\cite[Notation~2.2.4, Remark~2.2.8]{KX}}]
Let $\alpha\in [0,1)$. A {\it differential module} $M$ over $K\langle\alpha/t,t]\!]_0$ is a finite free $K\langle\alpha/t,t]\!]_0$-module equipped with an action of $d$.

Let $m$ be the rank of $M$. For $r\in (0,-\log{\alpha}]$ and $i\in\{1,\dots,m\}$, define
\[
f_i(M,r)=-\log{IR(M\otimes_{K\langle\alpha/t,t]\!]_0} F_{e^{-r}};i)}+r.
\]
We also define
\[
f_i(M,0)=-\log{IR(M\otimes_{K\langle\alpha/t,t]\!]_0}\mathcal{E};i)}.
\]
\end{dfn}

We recall only a few properties of the $f_i$'s, which will be used in this paper.

\begin{thm}[{\cite[Remarks~11.3.4, 11.6.5]{pde}}]\label{thm:cont}
Let $\alpha\in [0,1)$ and $M$ a differential module over $K\langle\alpha/t,t]\!]_0$ of rank $m$. For $i=1,\dots,m$, the function $f_i(M,r)$ on $[0,-\log{\alpha}]$ is continuous and piecewise affine with finitely many different slopes.
\end{thm}

\begin{thm}[{\cite[Theorem~2.3.9, Remark~2.3.11]{KX}, cf. \cite[Theorem~12.5.2, Remark~12.5.3]{pde}}]\label{thm:dec}
Let $\alpha\in [0,1)$ and $M$ a differential module over $K\langle\alpha/t,t]\!]_0$ of rank $m$. Suppose that for some $i\in\{1,\dots,m-1\}$, we have $f_i(M,0)>f_{i+1}(M,0)$. Then, $M\otimes_{K\langle\alpha/t,t]\!]_0}K\langle\alpha'/t,t]\!]_0$ for some $\alpha'\in [\alpha,1)$ admits a direct sum decomposition separating the first $i$ intrinsic generic subsidiary radii of $M\otimes_{K\langle\alpha/t,t]\!]_0} F_{\rho}$ for $\rho\in [\alpha',1)$ and $M\otimes_{K\langle\alpha/t,t]\!]_0}\mathcal{E}$.
\end{thm}
\section{Proof of Main Theorem}

In this section, we still assume that $K$ is discretely valued.

We first prepare some lemmata.

\begin{lem}[{cf. \cite[Proposition~5.2.5]{Chr}}]\label{lem:A}
Let $\alpha\in [0,1)$ and $M$ a differential module over $K\langle\alpha/t,t]\!]_0$ of rank $m$. If $IR(M\otimes_{K\langle\alpha/t,t]\!]_0}\mathcal{E};m)<1$, then
\[
(M\otimes_{K\langle\alpha/t,t]\!]_0}\mathcal{R})^{\nabla=0}=0.
\]
\end{lem}
\begin{proof}
We abbreviate $\otimes_{K\langle\alpha/t,t]\!]_0}$ as $\otimes$. Suppose the contrary, i.e., $(M\otimes\mathcal{R})^{\nabla=0}\neq 0$. We choose $\eta\in [\alpha,1)$ such that $(M\otimes K\langle\eta/t,t\})^{\nabla=0}\neq 0$. Since, for $\rho\in [\eta,1)$, we have
\[
(M\otimes K\langle\eta/t,t\})^{\nabla=0}\subset (M\otimes F_{\rho})^{\nabla=0},
\]
$M\otimes F_{\rho}$ has a non-zero trivial subobject. In particular, $IR(M\otimes F_{\rho};m)=1$ for $\rho\in [\eta,1)$ by Lemma~\ref{lem:exact}. By the continuity of $f_m$ (Theorem~\ref{thm:cont}), we have
\[
IR(M\otimes\mathcal{E};m)=\exp{(-f_m(M,0))}=\lim_{r\downarrow 0}\exp{(-f_m(M,r))}=\lim_{r\downarrow 0}{IR(M\otimes F_{e^{-r}};m)}e^{-r}=1,
\]
which is a contradiction.
\end{proof}

\begin{lem}\label{lem:B}
Let $\alpha\in [0,1)$ and $f\in K\langle\alpha/t,t\}$, $f\neq 0$. Then, the following are equivalent.
\begin{enumerate}
\item[(i)] $f\in K\langle\alpha/t,t]\!]_0$.
\item[(ii)] $f$ has a finite number of zeroes on $\alpha\le |t|<1$, i.e.,
\[
\#\{t\in K^{\mathrm{alg}};\alpha\le |t|<1, f(t)=0\}<\infty,
\]
where $K^{\mathrm{alg}}$ is an algebraic closure of $K$.
\end{enumerate}
\end{lem}
\begin{proof}
(i)$\Rightarrow$(ii) Since $K$ is discretely valued, any element of $K\langle\alpha/t,t]\!]_0$ is written as the product of $K\langle\alpha/t,t]\!]_0^{\times}$ and a polynomial over $K$. Since $u\in K\langle\alpha/t,t]\!]_0^{\times}$ has no zeroes on $\alpha\le |t|<1$, we obtain the assertion.

\noindent (ii)$\Rightarrow$(i) Recall that Newton polygon of $f=\sum{a_it^i}\in K\langle\alpha/t,t\}$ with $a_i\in K$ is the boundary of the lower convex hull of the set $(-i,v(a_i))$ where $i\in\mathbb{Z}$ such that $a_i\neq 0$, retaining only those slopes within $(0,-\log{\alpha}]$ (\cite[Definition~8.2.2]{pde}). The condition~(ii) implies the finiteness of the width of Newton polygon of $f$. Hence, for all $j\ge i$, $v(a_j)\ge v(a_i)$, in particular, $f\in K\langle\alpha/t,t]\!]_0$.
\end{proof}

\begin{cor}\label{cor:C}
Let $\alpha\in [0,1)$ and $M$ a differential module over $K\langle\alpha/t,t]\!]_0$ of rank $1$. Then,
\[
(M\otimes_{K\langle\alpha/t,t]\!]_0}\mathcal{E}^{\dagger})^{\nabla=0}=(M\otimes_{K\langle\alpha/t,t]\!]_0}\mathcal{R})^{\nabla=0}.
\]
\end{cor}
\begin{proof}
Since we have
\[
(M\otimes_{K\langle\alpha/t,t]\!]_0}\mathcal{E}^{\dagger})^{\nabla=0}\subset (M\otimes_{K\langle\alpha/t,t]\!]_0}\mathcal{R})^{\nabla=0}
\]
and both $K$-vector spaces are at most one dimensional, we have only to prove that if $(M\otimes_{K\langle\alpha/t,t]\!]_0}\mathcal{R})^{\nabla=0}\neq 0$, then $(M\otimes_{K\langle\alpha/t,t]\!]_0}\mathcal{E}^{\dagger})^{\nabla=0}\neq 0$. Let $e$ be a basis of $M$ and $g\in K\langle\alpha/t,t]\!]_0$ such that $D(e)=ge$. We may assume that $g\neq 0$. By assumption, there exists a non-zero $f\in\mathcal{R}$ such that $e\otimes f\in (M\otimes_{K\langle\alpha/t,t]\!]_0}\mathcal{R})^{\nabla=0}$. We have only to prove that $f\in\mathcal{E}^{\dagger}$. Let $\eta\in [\alpha,1)$ such that $f\in K\langle\eta/t,t\}$. Since $g$ has a finite number of zeroes on $\eta\le |t|<1$ by Lemma~\ref{lem:B}, so does $f$ by the relation
\[
df/dt=-gf,
\]
hence, $f\in K\langle\eta/t,t]\!]_0\subset\mathcal{E}^{\dagger}$ by Lemma~\ref{lem:B}.
\end{proof}

\begin{proof}[{Proof of Main Theorem}]
The case $m'=0$ is trivial and the case $m'=2$ is a special case of Dwork's theorem~\ref{thm:Dw}. Hence, we may assume that $m'=1$. If $IR(M\otimes_{K[\![t]\!]_0}\mathcal{E})=1$, then $M\otimes_{K[\![t]\!]_0}K\{t\}$ is trivial by Dwork's transfer theorem (\cite[Theorem~9.6.1]{pde}), which contradicts to $m=1$. Hence, $IR(M\otimes_{K[\![t]\!]_0}\mathcal{E};1)<1$. Since
\[
0\neq (M\otimes_{K[\![t]\!]_0}K\{t\})^{\nabla=0}\subset (M\otimes_{K[\![t]\!]_0}\mathcal{R})^{\nabla=0},
\]
$IR(M\otimes_{K[\![t]\!]_0}\mathcal{E};2)=1$ by Lemma~\ref{lem:A}. Thus,
\[
f_1(M,0)>f_2(M,0).
\]
By Theorem~\ref{thm:dec}, for some $\alpha\in [0,1)$, there exists a direct sum decomposition
\[
M\otimes_{K[\![t]\!]_0}K\langle\alpha/t,t]\!]_0=M_1\oplus M_2
\]
separating the first intrinsic generic (subsidiary) radius of $M\otimes_{K[\![t]\!]_0} F_{\rho}$ for $\rho\in [\alpha,1)$ and $M\otimes_{K[\![t]\!]_0}\mathcal{E}$. That is, we have
\[
\begin{cases}
f_1(M_1,r)=f_1(M,r),\\
f_1(M_2,r)=f_2(M,r)
\end{cases}
\]
for $r\in [0,-\log{\alpha}]$. We fix such an $\alpha$. Since
\[
IR(M_1\otimes_{K[\![t]\!]_0}\mathcal{E};1)=\exp{(-f_1(M_1,0))}=\exp{(-f_1(M,0))}=IR(M\otimes_{K[\![t]\!]_0}\mathcal{E};1)<1,
\]
we have $(M_1\otimes_{K\langle\alpha/t,t]\!]_0}\mathcal{R})^{\nabla=0}=0$ by Lemma~\ref{lem:A}. Hence, we have
\[
(M\otimes_{K[\![t]\!]_0}\mathcal{R})^{\nabla=0}=(M_1\otimes_{K\langle\alpha/t,t]\!]_0}\mathcal{R})^{\nabla=0}\oplus (M_2\otimes_{K\langle\alpha/t,t]\!]_0}\mathcal{R})^{\nabla=0}=(M_2\otimes_{K\langle\alpha/t,t]\!]_0}\mathcal{R})^{\nabla=0}=M_2^{\nabla=0},
\]
where the last equality follows from Corollary~\ref{cor:C}. Therefore,
\begin{align*}
&(M\otimes_{K[\![t]\!]_0}K\{t\})^{\nabla=0}=(M\otimes_{K[\![t]\!]_0}K\{t\})^{\nabla=0}\cap (M\otimes_{K[\![t]\!]_0}\mathcal{R})^{\nabla=0}=\left((M\otimes_{K[\![t]\!]_0}K\{t\})^{\nabla=0}\cap M_2^{\nabla=0}\right)\\
\subset&\left((M\otimes_{K[\![t]\!]_0}K\{t\})^{\nabla=0}\cap (M\otimes_{K[\![t]\!]_0}K\langle\alpha/t,t]\!]_0)^{\nabla=0}\right)=(M\otimes_{K[\![t]\!]_0}(K\{t\}\cap K\langle\alpha/t,t]\!]_0))^{\nabla=0}\\
=&(M\otimes_{K[\![t]\!]_0}K[\![t]\!]_0)^{\nabla=0}=M^{\nabla=0},
\end{align*}
i.e., $(M\otimes_{K[\![t]\!]_0}K\{t\})^{\nabla=0}\subset M^{\nabla=0}$, and the converse inclusion is trivial.
\end{proof}

As a final remark, we note that we can easily deduce the following generic analogue of Conjecture~\ref{conj:Dw} from the generic version of Dwork's theorem.

Let
\[
\tau:\mathcal{E}\to\mathcal{E}[\![X-t]\!]_0;f\mapsto\sum_{n=0}^{\infty}\frac{1}{n!}\frac{d^nf}{dt^n}(X-t)^n
\]
be a ring homomorphism (\cite[Proposition~0.1]{CT}).

\begin{prop}
Let $V$ be a differential module over $\mathcal{E}$. Let
\[
m':=\dim_{\mathcal{E}}(\tau^*V\otimes_{\mathcal{E}[\![X-t]\!]_0}\mathcal{E}\{X-t\})^{\nabla=0}\in \{0,\dots,\dim_{\mathcal{K}}V\}.
\]
Then,
\[
(\tau^*V\otimes_{\mathcal{E}[\![X-t]\!]_0}\mathcal{E}[\![X-t]\!]_{m'-1})^{\nabla=0}=(\tau^*V\otimes_{\mathcal{E}[\![X-t]\!]_0}\mathcal{E}\{X-t\})^{\nabla=0}.
\]
\end{prop}
\begin{proof}
We have only to prove that the RHS is contained in the LHS. By Theorem~\ref{thm:local}, we may assume that $V$ has pure radii $r$. Then,
\[
m'=
\begin{cases}
\dim_{\mathcal{K}}V&\text{if }r=1,\\
0&\text{if }0<r<1
\end{cases}
\]
by the geometric interpretation of the generic radius of convergence (\cite[Proposition~9.7.5]{pde}; Though the proposition treats only differential modules over $F_{\rho}$, the same proof works for differential modules over $\mathcal{E}$.). Hence, we may assume that $r=1$. Then, the assertion is nothing but the generic version of Dwork's theorem, i.e., Theorem~\ref{thm:Dw} for $K=\mathcal{E}$, $t=X-t$, and $M=\tau^*V$.
\end{proof}
\section{Appendix 1: Proof of theorem \ref{thm:dec}}

Theorem~\ref{thm:dec} plays a crucial role in the proof of Main Theorem. However, the proof of Theorem~\ref{thm:dec} in \cite{KX} is referred to \cite{pde} where the proof is left as an exercise (see Remark~12.5.3 loc. cit.). In this section, we will give a proof of Theorem~\ref{thm:dec} for the reader in a self-contained manner admitting some basic facts in \cite{pde}.

Throughout this section, we assume that $K$ is discretely valued.
\subsection*{Key Lemma}

We reduce Theorem~\ref{thm:dec} to the following lemma.

\begin{lem}[Key Lemma]\label{lem:main}
Let $M$ be a finite differential module over $K\langle\alpha/t,t]\!]_0$ of rank $m$. Assume that there exists $i\in\{0,\dots,m-1\}$ such that
\[
f_i(M,0)>f_{i+1}(M,0).
\]
Then, there exist $\alpha'\in [\alpha,1)$ and a differential submodule $M''$ of $M\otimes_{K\langle\alpha/t,t]\!]_0}K\langle\alpha'/t,t]\!]_0$ of rank $m-i$ such that
\[
f_{i+j}(M,0)=f_j(M'',0)\text{ for }j=1,\dots,m-i.
\]
\end{lem}

\begin{proof}[{Proof of Theorem~\ref{thm:dec} assuming Lemma~\ref{lem:main}}]
We may freely replace $M$ by $M\otimes_{K\langle\alpha/t,t]\!]_0}K\langle\alpha'/t,t]\!]_0$ for any $\alpha'\in [\alpha,1)$. Hence, by Lemma~\ref{lem:main}, we may assume that there exists a differential submodule $M''\subset M$ of rank $m-i$ such that
\[
f_{i+j}(M,0)=f_j(M'',0)\text{ for }j=1,\dots,m-i.\tag{a}
\]
Denote by $M\spcheck$ the dual of $M$. Since $f_j(M,0)=f_j(M\spcheck,0)$ for all $j$ (\cite[Lemma~6.2.8~(b)]{pde}), by Lemma~\ref{lem:main} again, there exists a differential submodule $\mathcal{M}\subset M\spcheck$ of rank $m-i$ such that
\[
f_j(\mathcal{M},0)=f_{i+j}(M\spcheck,0)=f_{i+j}(M,0)\text{ for }j=1,\dots,m-i.
\]
Let $(\ ,\ ):M\otimes_{K\langle\alpha/t,t]\!]_0}M\spcheck\to K\langle\alpha/t,t]\!]_0$ be the canonical perfect pairing and
\[
M':=\{v\in M;(v,w)=0\ \forall w\in\mathcal{M}\}.
\]
Since $M'\cong (M\spcheck/\mathcal{M})\spcheck$, $M'$ is a differential submodule of $M$ of rank $i$ such that
\[
f_j(M',0)=f_j(M,0)\text{ for }j=1,\dots,i.\tag{b}
\]
Since
\[
f_i(M',0)=f_i(M,0)>f_{i+1}(M,0)=f_1(M'',0)\tag{c}
\]
by (a) and (b), we have $\mathfrak{IR}(M'\otimes_{K\langle\alpha/t,t]\!]_0}\mathcal{E})\cap\mathfrak{IR}(M''\otimes_{K\langle\alpha/t,t]\!]_0}\mathcal{E})=\emptyset$, which implies $M'\cap M''=\{0\}$. By comparing ranks, $M'\oplus M''=M$. By (c) and the continuity of the $f_i$'s (Theorem~\ref{thm:cont}), we may assume that
\[
f_i(M',r)>f_1(M'',r)\text{ for }r\in [0,-\log{\alpha}]
\]
by choosing $\alpha$ sufficiently close to $1$ if necessary. By the definition of the $f_i$'s,
\[
f_j(M,r)=f_j(M',r)\text{ for }r\in [0,-\log{\alpha}]\text{ and }j=1,\dots,i
\]
which implies the assertion.
\end{proof}
\subsection*{Notation on differential rings}

A {\it differential ring} $(R,d_R)$ is a commutative ring $R$ with a derivation $d_R:R\to R$. A {\it homomorphism} of differential rings $f:(R,d_R)\to (S,d_S)$ is a ring homomorphism $f:R\to S$ such that $d_S\circ f=f\circ d_R$. When no confusion arises, we write $(R,d)$, $(S,d)$ for $(R,d_R)$, $(S,d_S)$.

A {\it differential $(R,d_R)$-module} $(M,D)$ is an $R$-module $M$ (we do not assume a freeness in the following) with a differential $D:M\to M$ satisfying
\[
D(r\cdot m)=dr\cdot m+r\cdot D(m),\ r\in R,m\in M.
\]
We denote by $\mathrm{Mod}(R,d_R)$ the category of differential $(R,d_R)$-modules. A {\it differential $(R,d_R)$-submodule} $(M_0,D_0)$ of $(M,D)$ is an $R$-submodule $M_0$ of $M$ together with a differential $D_0:M_0\to M_0$ such that $D|_{M_0}=D_0$. For simplicity, we write $(M_0,D)$ for $(M_0,D_0)$. Note that a differential submodule of $(M,D)$ is nothing but a subobject of $(M,D)$. In the following, we consider the situation where $R$ is endowed with multiple derivations (see Example~\ref{ex:Frob}), hence, we basically do not omit derivations or differentials.

We define the {\it pull-back} and {\it push-out} of a homomorphism of differential rings $f:(R,d_R)\to (S,d_S)$
\begin{gather*}
f^*:\mathrm{Mod}(R,d_R)\to \mathrm{Mod}(S,d_S),\\
f_*:\mathrm{Mod}(S,d_S)\to \mathrm{Mod}(R,d_R)
\end{gather*}
as follows. For $(M,D)\in \mathrm{Mod}(R,d_R)$, let $f^*(M,D)=(f^*M,f^*D)$ where $f^*M=M\otimes_RS$, $f^*D(m\otimes s)=D(m)\otimes s+m\otimes ds$ for $m\in M$, $s\in S$. For $(N,D)\in \mathrm{Mod}(S,d_S)$, let $f_*(N,D)=(f_*N,f_*D)$ where $f_*N=N$ whose $R$-structure is defined via $f$, and $f_*D(n)=D(n)$.


Let $(R,d_R)$, $(S,d_S)$, $(T,d_T)$, $(U,d_U)$ be differential rings. Assume that
\begin{equation}\label{xy:car}
\begin{xymatrix}{
S\ar[r]^{\beta}&U\\
R\ar[r]^{\gamma}\ar[u]^{\alpha}&T\ar[u]^{\delta}
}\end{xymatrix}
\end{equation}
is a cocartesian diagram of commutative rings such that $\alpha$, $\beta$, $\gamma$, and $\delta$ induce morphisms of differential rings. Then, we define the derivation $d_{S\otimes T}$ on $S\otimes_R T$ by
\[
d_{S\otimes T}(s\otimes t)=d_S(s)\otimes t+s\otimes d_T(t)
\]
and assume that the isomorphism
\[
\beta\otimes\gamma:S\otimes_RT\to U
\]
induces an isomorphism of differential rings $(S\otimes_R T,d_{S\otimes T})\to(U,d_U)$. In this case, we say that (\ref{xy:car}) is a {\it cocartesian diagram of differential rings}. Note that the functors
\[
\gamma^*\alpha_*,\delta_*\beta^*:\mathrm{Mod}(S,d_S)\to \mathrm{Mod}(T,d_T)
\]
are naturally isomorphic to each other. In the following, we identify $\gamma^*\alpha_*$ as $\delta_*\beta^*$.
\subsection*{Frobenius}
In this subsection, we assume the following.

\begin{ass}\label{ass:Frob}
Let $\varphi:(R',d')\to (R,d')$ be a homomorphism of differential rings. Assume that there exists $t\in R^{\times}$ such that $d't/t\in\varphi(R')$, and $\{1,t,\dots,t^{p-1}\}$ is a basis of $R$ as an $R'$-module, where $R$ is regarded as an $R'$-module via $\varphi$.
\end{ass}

Note that $\varphi$ is flat and we regard $R'$ as a subring of $R$. For an $R$-module $M$, we define an $R'$-module homomorphism $\psi_i:\varphi_*M\to\varphi_*M$ for $i\in\mathbb{N}$ by the map induced by the multiplication by $t^i$ via $\varphi$. Note that $\psi_0=\mathrm{id}_{\varphi_*M}$. Also note that if $(M,D')\in \mathrm{Mod}(R,d')$, then we have $\varphi_*D'\circ\psi_i=(id't/t)\psi_{i}+\psi_i\circ \varphi_*D'$ for $i\ge 1$.

\begin{ex}\label{ex:Frob}
\begin{enumerate}
\item[(i)](cf. \cite[Remark~10.3.5]{pde}) Let $\alpha\in (0,1)$ and $(K\langle\alpha/t,t]\!]_0,d)$ as in Definition~\ref{dfn:ring}. We apply the same construction as $(K\langle\alpha/t,t]\!]_0,d)$ replacing $\alpha$, $t$ by $\alpha^p$, $t^p$ to obtain $(K\langle\alpha^p/t^p,t^p]\!]_0,d')$. That is,
\[
K\langle\alpha^p/t^p,t^p]\!]_0:=\left\{\sum_{i\in\mathbb{Z}}a_i(t^p)^i;a_i\in K, \lim_{i\to-\infty}{|a_i|\alpha^{pi}}=0,\ \sup_{i\in\mathbb{Z}}{|a_i|}<\infty\right\}
\]
and $d'=d/d(t^p)$. We endow $K\langle\alpha/t,t]\!]_0$ with another derivation $d'=(pt^{p-1})^{-1}d$.

Let
\[
\varphi:K\langle\alpha^p/t^p,t^p]\!]_0\to K\langle\alpha/t,t]\!]_0;\sum_{i\in\mathbb{Z}}a_i(t^p)^i\mapsto\sum_{i\in\mathbb{Z}}a_it^{pi}
\]
be the ring homomorphism, which induces a homomorphism of differential rings $(K\langle\alpha^p/t^p,t^p]\!]_0,d')\to (K\langle\alpha/t,t]\!]_0,d')$. Then, $(\varphi,t)$ satisfies Assumption~\ref{ass:Frob}. Indeed, let $\sum_{i}a_it^{i}\in K\langle\alpha/t,t]\!]_0$ with $a_i\in K$. Then, we can write
\[
\sum_{i\in\mathbb{Z}}a_it^{i}=\sum_{j=0}^{p-1}\left(\sum_{\substack{i\ge 0\\ i\equiv j \mod{p}}}{a_i(t^p)^{(i-j)/p}}+\sum_{\substack{i>0\\ i\equiv -j \mod{p}}}{a_{-i}(t^p)^{-(i+j)/p}}\right)t^j,
\]
where the coefficient of $t^j$ belongs to $K\langle\alpha^p/t^p,t^p]\!]_0$. We prove the linear independence of $\{1,\dots,t^{p-1}\}$ as follows. Assume that we have a relation $a^{(0)}t^0+\dots+a^{(p-1)}t^{p-1}=0$ with $a^{(j)}\in K\langle\alpha^p/t^p,t^p]\!]_0$. Write $a^{(j)}=\sum_{n}a_{np+j}(t^p)^n$ with $a_{np+j}\in K$. Then,
\[
0=\sum_{j=0}^{p-1}a^{(j)}t^j=\sum_{j=0}^{p-1}\sum_{n\in\mathbb{Z}}a_{np+j}t^{np+j},
\]
which implies $a_{np+j}=0$ for all $n$, $j$.
\item[(ii)] (\cite[Definition~10.3.1]{pde}) Let $(\mathcal{E},d)$ be as in Definition~\ref{dfn:Amice}. We apply the same construction as $(\mathcal{E},d)$ replacing $t$ by the new variable $t^p$ to obtain $(\mathcal{E}',d')$. That is, $\mathcal{E}'$ is the fraction field of the $p$-adic completion of $\mathcal{O}_K[\![t^p]\!][(t^p)^{-1}]$ and $d'=d/d(t^p)$. We endow $\mathcal{E}$ with another derivation $d'=(pt^{p-1})^{-1}d$. Let $\Phi:\mathcal{E}'\to\mathcal{E}$ be the inclusion, which induces a homomorphism of differential rings $(\mathcal{E}',d')\to (\mathcal{E},d')$ Then, $(\Phi,t)$ satisfies Assumption~\ref{ass:Frob}. Note that there exists a cocartesian diagram of differential rings
\[
\begin{xymatrix}{
(K\langle\alpha/t,t]\!]_0,d')\ar[r]^(.6)H&(\mathcal{E},d')\\
(K\langle\alpha^p/t^p,t^p]\!]_0,d')\ar[r]^(.6){h}\ar[u]^{\varphi}&(\mathcal{E}',d').\ar[u]^{\Phi}
}\end{xymatrix}
\]
\end{enumerate}
\end{ex}

\begin{lem}\label{lem:des}
Let $(M,D')\in \mathrm{Mod}(R,d')$ and $(N',\varphi_*D')\subset (\varphi_*M,\varphi_*D')\in\mathrm{Mod}(R',d')$ a subobject. Then, the following are equivalent.
\begin{enumerate}
\item[(i)] There exists a unique subobject $(N,D')\subset (M,D')$ such that
\[
(N',\varphi_*D')=(\varphi_*N,\varphi_*D').
\]
\item[(ii)] For $i=0,\dots,p-1$, $\psi_i(N')\subset N'$.
\end{enumerate}
\end{lem}
\begin{proof}
(i)$\Rightarrow$(ii)
\[
\psi_i(N')=\psi_i(\varphi_*N)=\varphi_*(t^iN)\subset\varphi_*N=N'.
\]
\noindent (ii)$\Rightarrow$(i) The uniqueness is obvious. We prove the existence. The $R'$-module structure on $\varphi_*M$ extends to an $R$-module structure by assumption. Indeed, the multiplication by $t^i$ is defined by $\psi_i$. Let $N$ be the $R$-module defined in this way. Then, $N$ is an $R$-submodule of $M$ and $D'(N)\subset N$. It is obvious $(N,D')$ satisfies the condition.
\end{proof}

\begin{dfn}\label{lem:G}
Let $(M,D')\in \mathrm{Mod}(R,d')$ and $(N',\varphi_*D')\subset (\varphi_*M,\varphi_*D')\in\mathrm{Mod}(R',d')$ a subobject. We define the subobject $\mathcal{G}_{\varphi}(N',\varphi_*D')$ of $(\varphi_*M,\varphi_*D')$ by
\[
\mathcal{G}_{\varphi}(N',\varphi_*D'):=\left(\sum_{i=0}^{p-1}\psi_i(N'),\varphi_*D'\right).
\]
Note that since $\psi_i(\sum_{i'=0}^{p-1}\psi_{i'}(N'))\subset\sum_{i'=0}^{p-1}\psi_{i'}(N')$, there exists, by Lemma~\ref{lem:des}, a unique subobject $(N,D')$ of $(M,D')$ such that $(\varphi_*N,\varphi_*D')=\mathcal{G}_{\varphi}(N',\varphi_*D')$.
\end{dfn}

Let
\[
\begin{xymatrix}{
(R,d')\ar[r]^H&(\mathcal{R},d')\\
(R',d')\ar[r]^{h}\ar[u]^{\varphi}&(\mathcal{R}',d')\ar[u]^{\Phi}
}\end{xymatrix}
\]
be a cocartesian diagram of differential rings. Assume that there exists $t\in R$ such that $(\varphi,t)$ and $(\Phi,H(t))$ satisfy Assumption~\ref{ass:Frob}. 

\begin{lem}\label{lem:G-pull}
Let notation be as above. Let $(M,D')\in \mathrm{Mod}(R,d')$ and $(N',\varphi_*D')\subset (\varphi_*M,\varphi_*D')$ a subobject. Then, $(h^*N',h_*\varphi^*D')\subset (h^*\varphi_*M,h^*\varphi_*D')$ is a subobject, and as subobjects of $(h^*\varphi_*M,h^*\varphi_*D')$,
\[
h^*\mathcal{G}_{\varphi}(N',\varphi_*D')=\mathcal{G}_{\Phi}(h^*N',h^*\varphi_*D').
\]
\end{lem}
\begin{proof}
Since $\psi_i$ and $h^*$ commute, we obtain the equality as $\mathcal{R}'$-modules, which implies the assertion.
\end{proof}

\subsection*{Local computation}

In this subsection, we use the notation in Example~\ref{ex:Frob}~(ii). We also use the following notation.

For a non-negative real number $s$, let $\phi(s):=\inf{\{s^p,s/p\}}$. Note that $\phi(s)=s/p$ if $s\le\omega$ and $\phi(s)=s^p$ if $s\ge\omega$, and $\phi$ is strictly increasing on $(0,\infty)$. For $S=\{s_1,\dots,s_m\}$ a multiset of non-negative real numbers, let
\[
\phi(S):=\bigcup_{i=1}^m\begin{cases}
\{s_i^p,\omega^p\ (p-1\text{ times})\}&\text{if }s_i>\omega,\\
\{s_i/p\ (p\text{ times})\}&\text{if }s_i\le\omega.
\end{cases}
\]
Then, for multisets $S$, $S'$, we have $\phi(S)=\phi(S')$ if and only if $S=S'$. This follows from the fact that the elements of $\phi(S)$ strictly greater than $\omega^p$ coincides with the $p$-th power of the elements of $S$ strictly greater than $\omega$.

\begin{lem}\label{lem:IR}
For $(M,D)\in\mathrm{Mod}(\mathcal{E},d)$, we define $D':=(pt^{p-1})^{-1}D$ so that $(M,D')\in\mathrm{Mod}(\mathcal{E},d')$.
\begin{enumerate}
\item[(i)](\cite[Theorem~10.5.1]{pde}) Let $(M,D)\in\mathrm{Mod}(\mathcal{E},d)$. Then,
\[
\mathfrak{IR}(\Phi_*(M,D'))=\phi(\mathfrak{IR}(M,D)).
\]
\item[(ii)] Let $(M,D_1)$, $(N,D_2)\in \mathrm{Mod}(\mathcal{E},d)$. If there exists an isomorphism
\begin{equation}\label{eq:iso}
\Phi_*(M,D'_1)\cong \Phi_*(N,D'_2)
\end{equation}
in $\mathrm{Mod}(\mathcal{E}',d')$, then
\[
\mathfrak{IR}(M,D_1)=\mathfrak{IR}(N,D_2).
\]
\end{enumerate}
\end{lem}
\begin{proof}
(ii) The isomorphism (\ref{eq:iso}) implies that $\phi(\mathfrak{IR}(M,D_1))=\phi(\mathfrak{IR}(N,D_2))$ by (i), hence, $\mathfrak{IR}(M,D_1)=\mathfrak{IR}(N,D_2)$.
\end{proof}

Let $(M,D)\in \mathrm{Mod}(\mathcal{E},d)$. By Theorem~\ref{thm:local}, there exists a unique decomposition
\[
(M,D)=(M_{\le s},D)\oplus (M_{>s},D)
\]
such that any element of $\mathfrak{IR}(M_{\le s})$ (resp. $\mathfrak{IR}(M_{>s})$) is less than or equal to $s$ (resp. strictly greater than $s$). Indeed, let $M_{\le s}$ (resp. $M_{>s}$) be the direct sum of the $M_r$'s with $r\le s$ (resp. $r>s$). 
For $(\Phi_*M,\Phi_*D')\in \mathrm{Mod}(\mathcal{E}',d')$, we have a similar decomposition
\[
(\Phi_*M,\Phi_*D')=((\Phi_*M)_{\le \phi(s)},\Phi_*D')\oplus((\Phi_*M)_{>\phi(s)},\Phi_*D').
\]
The latter decomposition recovers the first one as follows.

\begin{lem}[cf. {\cite[Proof of Theorem~12.2.2]{pde}}]\label{lem:push}
Let notation be as above. As subobjects of $(\Phi_*M,\Phi_*D')$, we have
\[
\Phi_*(M_{>s},D')=\mathcal{G}_{\Phi}((\Phi_*M)_{>\phi(s)},\Phi_*D').
\]
\end{lem}
\begin{proof}
We have only to prove the equality as sets. By Theorem~\ref{thm:local}, we may assume that $M$ has pure radii $t\in (0,\infty)$. We separate the cases as follows.

a-1. $t\le s\le \omega$, a-2. $s=\omega<t$, a-3. $s<\omega\le t$, a-4. $s<t\le\omega$.

b-1. $t\le \omega<s$, b-2. $\omega<t\le s$, b-3. $\omega<s<t$.

When $s<\omega$, either a-1, a-3, or a-4 occurs. When $s=\omega$, either a-1 or a-2 occurs. When $s>\omega$, either b-1, b-2, or b-3 occurs.

By using Lemma~\ref{lem:IR}~(i),
\[
(M_{>s},(\Phi_*M)_{>\phi(s)})=
\begin{cases}
(0,0)&\text{the case a-1 or b-1},\\
(M,\Phi_*M)&\text{the case a-3 or a-4}.
\end{cases}
\]
In these cases, the assertion is obvious. In the rest of the cases, i.e., a-2, b-2, or b-3, we have $\omega<t$, hence, there exists a Frobenius antecedent $M'$ of $M$, that is, $(M,D')\cong(\Phi^*M',\Phi^*D')$ (\cite[Theorem~10.4.2]{pde}). We may identify $(\Phi_*M,\Phi_*D')$ with $(\oplus_{i=0}^{p-1}\psi_i(M'),\Phi_*D')$. By the proof of \cite[Theorem~10.5.1]{pde}, $(\psi_0(M'),\Phi_*D')$ has pure radii $t^p$ and $(\psi_i(M'),\Phi_*D')$ for $1\le i\le p-1$ has pure radii $\omega^p$. Hence, we have
\[
(M_{>s},(\Phi_*M)_{>\phi(s)})=
\begin{cases}
(M,\psi_0(M'))&\text{the case a-2 or b-3},\\
(0,0)&\text{the case b-2}.
\end{cases}
\]
In these cases, we obtain the assertion by $\sum_{i=0}^{p-1}\psi_i\circ\psi_{0}(M')=\sum_{i=0}^{p-1}\psi_{i}(M')=M$.
\end{proof}
\subsection*{Hensel's lemma for twisted polynomials}
In this subsection, let notation be as follows. Let $F$ be a commutative ring with multiplicative norms $\{|\cdot|_{\alpha}\}_{\alpha\in I}$. We set $v^{\alpha}(\cdot)=-\log{|\cdot|_{\alpha}}$. A sequence $\{a_n\}_{n\in\mathbb{N}}$ of $F$ is {\it Cauchy} if it is a Cauchy sequence with respect to $|\cdot|_{\alpha}$ for all $\alpha\in I$. A sequence $\{a_n\}_{n\in\mathbb{N}}$ of $F$ is {\it converging} (to $a\in F$) if $\{a_n\}_{n\in\mathbb{N}}$ converges to $a$ with respect to $|\cdot|_{\alpha}$ for all $\alpha\in I$. We assume that $F$ is {\it Fr\'echet complete}, i.e., any Cauchy sequence of $F$ is converging.

Let $d:F\to F$ be a continuous derivation, i.e., $d$ is a derivation such that for any Cauchy sequence $\{a_n\}_{n\in\mathbb{N}}$, $\{da_n\}_{n\in\mathbb{N}}$ is Cauchy. Let $F\{T\}$ be the ring of twisted polynomials over $F$ (\cite[\S~5.5]{pde}). For $f=R_0+R_1T+\dots\in F\{T\}$, $R_i\in F$, we define Newton polygon of $f$ with respect to $v^{\alpha}$ as the boundary of the lower convex hull of the points $(-i,v^{\alpha}(R_i))$, $i\in\mathbb{N}$ (\cite[Definition~2.1.3]{pde}).

For $r\in\mathbb{R}$, we define 
\[
v^{\alpha}_r\left(\sum_ka_kT^k\right)=\inf_k\{v^{\alpha}(a_k)+kr\}.
\]
Let $r^{\alpha}_0=\inf_{x\in F,x\neq 0}{\{v^{\alpha}(d(x))-v^{\alpha}(x)\}}$.

\begin{ex}
Let notation be as in Example~\ref{ex:Frob}~(i). Then, the ring $K\langle\alpha/t,t]\!]_0$ is Fr\'echet complete with respect to Gaussian norms $\{|\cdot|_{\rho}\}_{\rho\in [\alpha,1]}$ (\cite[Proposition~8.2.5]{pde}). Moreover, $r^{\rho}_0=\log{\rho}$ (Definition~9.4.1 loc. cit.).
\end{ex}

The following is a slight generalization of Robba's analogue for differential operators of Hensel's lemma for a twisted polynomial over a complete non-archimedean valuation field (\cite[Th\'eor\`em~2.4]{Rob}). Fortunately, the original proof works in our situation.

\begin{prop}[cf. {\cite[Proposition~3.2.2]{ss}, \cite[Theorem~2.2.1]{pde}}]\label{prop:Hensel}
Let $r\in\mathbb{R}$ such that $r<r^{\alpha}_0$ for all $\alpha\in I$. Let
\[
R=R_0+R_1T+\dots+R_iT^i+\dots\in F\{T\}
\]
such that
\begin{gather*}
R_i\in F^{\times},\\
v^{\alpha}_r(R-R_iT^i)>v^{\alpha}_r(R_iT^i)\ \forall\alpha\in I.
\end{gather*}
Then, $R$ can be factored uniquely as $PQ$ (resp. $Q'P'$) where $P\in F\{T\}$ (resp. $P'\in F\{T\}$) has degree $\deg{(R)}-i$ and has all slopes with respect to $v^{\alpha}$ strictly less than $r$, $Q\in F\{T\}$ (resp. $Q'\in F\{T\}$) is degree $i$ whose leading term is $R_i$ and has all slopes with respect to $v^{\alpha}$ strictly greater than $r$, $v^{\alpha}_r(P-1)>0$, and $v_r^{\alpha}(Q-R_iT^i)>v_r^{\alpha}(R_iT^i)$ (resp. $v^{\alpha}_r(P'-1)>0$, and $v_r^{\alpha}(Q'-R_iT^i)>v_r^{\alpha}(R_iT^i)$).
\end{prop}
\begin{proof}
Since there exists a canonical isomorphism $F^{ad}\{T\}\cong F\{T\}^{op}$, where $F^{ad}$ is the ring $F$ with the derivation $-d$, we have only to prove the existence and uniqueness for the first decomposition $R=PQ$ (\cite[Remark~3.1.3]{ss}).

We first check the existence. Define sequences $\{P_l\}$, $\{Q_l\}$ as follows. Define $P_0=1$ and $Q_0=R_{i}T^i$. Given $P_l$ and $Q_l$, write
\[
R-P_lQ_l=\sum_ka_kT^k,a_k\in F,
\]
then put
\[
X_l=\sum_{k\ge i}a_kT^{k-i}R_i^{-1},\ Y_l=\sum_{k<i}a_kT^k
\]
and set $P_{l+1}=P_l+X_l$, $Q_{l+1}=Q_l+Y_l$. Put $c_l^{\alpha}=v_r^{\alpha}(R-P_lQ_l)-ri-v^{\alpha}(R_i)$, so that $c_0^{\alpha}>0$. Suppose that
\begin{gather}
v_r^{\alpha}(P_l-1)\ge c_0^{\alpha},\label{1}\\
v_r^{\alpha}(Q_l-R_iT^i)\ge c_0^{\alpha}+ri+v^{\alpha}(R_i),\label{2}\\
c_l^{\alpha}\ge c_0^{\alpha}.\label{3}
\end{gather}
The assumption $r<r_0^{\alpha}$ implies that $v_r^{\alpha}(fg)=v_r^{\alpha}(f)+v_r^{\alpha}(g)$ for $f$, $g\in F\{T\}$ by \cite[Lemma~3.1.5]{ss}. Hence, we have
\[
v_r^{\alpha}(R-P_lQ_l)=\inf{\{v_r^{\alpha}(X_lR_iT^i),v_r^{\alpha}(Y_l)\}}=\inf{\{v_r^{\alpha}(X_l)+ri+v^{\alpha}(R_i),v_r^{\alpha}(Y_l)\}},
\]
in particular,
\begin{gather}
v_r^{\alpha}(X_l)\ge v_r^{\alpha}(R-P_lQ_l)-ri-v^{\alpha}(R_i)=c_l^{\alpha},\label{4}\\
v_r^{\alpha}(Y_l)\ge v_r^{\alpha}(R-P_lQ_l)=c_l^{\alpha}+ri+v^{\alpha}(R_i).\label{5}
\end{gather}
We will check (\ref{1}), (\ref{2}), and (\ref{3}) for $l+1$. 

We have, by (\ref{1}), (\ref{3}), and (\ref{5}),
\[
v_r^{\alpha}(P_{l+1}-1)\ge \inf{\{v_r^{\alpha}(P_l-1),v_r^{\alpha}(X_l)\}}\ge\inf{\{c_0^{\alpha},c_l^{\alpha}\}}=c_0^{\alpha},
\]
by (\ref{2}), (\ref{3}), and (\ref{5}),
\begin{align*}
&v_r^{\alpha}(Q_{l+1}-R_iT^i)\ge\inf{\{v_r^{\alpha}(Q_l-R_iT^i),v_r^{\alpha}(Y_l)\}}\\
\ge&\inf{\{c_0^{\alpha}+ri+v^{\alpha}(R_i),c_l^{\alpha}+ri+v^{\alpha}(R_i)\}}=c_0^{\alpha}+ri+v^{\alpha}(R_i).
\end{align*}
Since
\[
R-P_{l+1}Q_{l+1}=X_l(R_iT^i-Q_l)+(1-P_l)Y_l-X_lY_l,
\]
we have, by (\ref{1}), (\ref{2}), (\ref{4}), and (\ref{5}),
\[
c_{l+1}^{\alpha}\ge\inf{\{c_l^{\alpha}+c_0^{\alpha}+ri+v^{\alpha}(R_i),c_0^{\alpha}+c_l^{\alpha}+ri+v^{\alpha}(R_i)\}}-ri-v^{\alpha}(R_i)=c_l^{\alpha}+c_0^{\alpha}.
\]
By induction on $l$, we deduce that $c_l^{\alpha}\ge (l+1)c_0^{\alpha}$. Moreover, each $P_l$ has degree at most $\deg{(R)}-i$, and each $Q_l-R_iT^i$ has degree at most $i-1$. For each $j\in\{0,\dots,\deg{(R)}-i\}$, let $P_{l,j}\in F$ denote the coefficient of $T^j$ in $P_l$. Then, the sequence $\{P_{l,j}\}_{l\in\mathbb{N}}$ is Cauchy, hence, converges to some element $P_j\in F$. Thus, we obtain the twisted polynomial $P=\sum_{0\le j\le\deg{(R)}-i}P_{j}T^j\in F\{T\}$. Similarly, by starting with the sequence $\{Q_l\}$, we obtain a twisted polynomial $Q\in F\{T\}$ of degree at most $i$. By construction, $P$ and $Q$ satisfy the desired properties.

We next check the uniqueness. Let $R=\widetilde{P}\widetilde{Q}$ be another factorization satisfying the condition. Let $X=R-P\widetilde{Q}$ and suppose that $X\neq 0$. Fix some $\alpha\in I$ and write $v_r$ for $v_r^{\alpha}$. For any $f\in F\{T\}$, the map $s\mapsto v_s(f)$ is piecewise affine with slopes non-negative integers. Hence, we may assume that
\[
v_{r'}(Q-R_iT^i)>v_{r'}(R_iT^i),\ v_{r'}(\widetilde{P}-1)>0
\]
for $r'$ in some left neighborhood of $r$. Then, for any $r'\le r$ sufficiently close to $r$, we have
\begin{equation}\label{eq:un1}
v_{r'}(X)=v_{r'}(\widetilde{Q}-Q+(\widetilde{P}-1)(\widetilde{Q}-Q))=v_{r'}(\widetilde{Q}-Q)
\end{equation}
and
\begin{equation}\label{eq:un2}
v_{r'}(X)=v_{r'}((P-\widetilde{P})R_iT^i+(P-\widetilde{P})(Q-R_iT^i))=v_{r'}((P-\widetilde{P})R_iT^i)=v_{r'}(P-\widetilde{P})+r'i+v(R_i).
\end{equation}
The left derivative of the function $r'\mapsto v_{r'}(X)$ at $r'=r$  is less than or equal to $i-1$ by (\ref{eq:un1}), and is greater than or equal to $i$ by (\ref{eq:un2}), which is a contradiction.
\end{proof}
\subsection*{Proof of Key Lemma}
We first prove in the case $f_i(M,0)>-\log{\omega}$. We may freely replace $\alpha$ by $\alpha'\in [\alpha,1)$. Since $\mathcal{E}^{\dagger}=\cup_{\alpha\in (0,1)}K\langle\alpha/t,t]\!]_0$ is a field by \cite[Definition~15.1.2]{pde} (recall that $K$ is discretely valued), $M\otimes_{K\langle\alpha/t,t]\!]_0}\mathcal{E}^{\dagger}$ admits a cyclic vector (\cite[Theorem~5.4.2]{pde}). Hence, we may assume that $M$ admits a cyclic vector, in particular, there exists an isomorphism
\[
M\cong K\langle\alpha/t,t]\!]_0\{T\}/K\langle\alpha/t,t]\!]_0\{T\}R
\]
for some twisted polynomial $R\in K\langle\alpha/t,t]\!]_0\{T\}$ of degree $m$.  We may also assume that each non-zero coefficient of $R$ is invertible in $K\langle\alpha/t,t]\!]_0$. We apply Proposition~\ref{prop:Hensel} to $R$ with $F=K\langle\alpha/t,t]\!]_0$, then we construct $M''$ by using the resulting decomposition of $R$. Let $\mathfrak{IR}(M,0)=\{s_1\le\dots\le s_m\}$. The assumptions $f_i(M,0)>-\log{\omega}$ and $f_i(M,0)>f_{i+1}(M,0)$ imply that $s_i<\omega$ and $s_i<s_{i+1}$ respectively. For a while, we regard $R$ as a twisted polynomial over $\mathcal{E}$. Recall that Newton polygon of $R$ in the sense of \cite[Definition 6.4.3]{pde} is obtained from Newton polygon of $R$ in the sense of this paper by omitting all slopes greater than or equal to $0$. Hence, by \cite[Corollary~6.5.4]{pde}, the multiset $\{\log{(s_j/\omega)};s_j<\omega\}$ coincides with the one obtained from the slope multiset of Newton polygon of $R$ by omitting all elements greater than or equal to $0$. Hence, Newton polygon of $R$ has a vertex whose $x$-coordinate is equal to $-m+i$. In particular, we have $R_{m-i}\neq 0$, hence, $R_{m-i}\in (K\langle\alpha/t,t]\!]_0)^{\times}$ by assumption. Moreover, if $r$ is any real number satisfying $-\log{(\omega/s_i)}<r<-\log{(\omega/s_{i+1})}$, then $v_r^1(R-R_{m-i}T^{m-i})>v_r^1(R_{m-i}T^{m-i})$, where $v^{\rho}(\cdot)=-\log{|\cdot|_{\rho}}$, with the notation as in the previous subsection. We fix $r$ satisfying
\[
-\log{(\omega/s_i)}<r<\inf{\{\log{\alpha},-\log{(\omega/s_{i+1})}\}}
\]
by choosing $\alpha$ sufficiently close to $1$ if necessary (note that $-\log{(\omega/s_i)}<0$ by assumption). By the continuity of $\rho\mapsto v^{\rho}(\cdot)$, we may assume that $v^{\rho}_r(R-R_{m-i}T^{m-i})>v_r^{\rho}(R_{m-i}T^{m-i})$ for $\rho\in [\alpha,1]$. By applying Proposition~\ref{prop:Hensel} to $R$, we obtain a decomposition $R=Q'P'$, where $P'$ (resp. $Q'$) as a twisted polynomial over $\mathcal{E}$ is of degree $i$ (resp. $m-i$) with slopes strictly less than $r$ (resp. strictly greater than $r$). Let $M''=K\langle\alpha/t,t]\!]_0\{T\}P'/K\langle\alpha/t,t]\!]_0\{T\}R$ be the differential submodule of $M$ of rank $m-i$. Then,
\[
M''\cong K\langle\alpha/t,t]\!]_0\{T\}/K\langle\alpha/t,t]\!]_0\{T\}Q',M/M''\cong K\langle\alpha/t,t]\!]_0\{T\}/K\langle\alpha/t,t]\!]_0\{T\}P'.
\]
By \cite[Theorem~6.5.3]{pde} again, the intrinsic generic subsidiary radii of $M''\otimes_{K\langle\alpha/t,t]\!]_0}\mathcal{E}$ (resp.\\ $M/M''\otimes_{K\langle\alpha/t,t]\!]_0}\mathcal{E}$) are strictly less than $\omega e^r$ (resp. strictly greater than $\omega e^r$). Therefore, we have $f_{i+j}(M,0)=f_j(M'',0)$ for $j=1,\dots,m-i$.

We next prove in the case $f_i(M,0)>-p^{-j}\log{\omega}$ for $j\in\mathbb{N}$ by induction on $j$. The case $j=0$ is already proved. We assume the case $j-1$. We use the notation in Example~\ref{ex:Frob}. Let $D':=(pt^{p-1})^{-1}D$ so that $(M,D')\in \mathrm{Mod}(K\langle\alpha/t,t]\!]_0,d')$. Since $\phi(\mathfrak{IR}(H^*M,H^*D))=\mathfrak{IR}(\Phi_*(H^*M,H^*D'))=\mathfrak{IR}(h^*(\varphi_*M,\varphi_*D'))$ by Lemma~\ref{lem:IR}~(i), there exists $i'\in\mathbb{N}$ such that 
\begin{gather}
f_{i'}(\varphi_*M,0)>f_{i'+1}(\varphi_*M,0),\label{6}\\
\phi(\exp{(-f_i(M,0))})=\exp{(-f_{i'}(\varphi_*M,0))}.\label{7}
\end{gather}
Indeed, let $\mathfrak{IR}(H^*M,H^*D)=\{s_1,\dots,s_m\}$, $\mathfrak{IR}(h^*\varphi_*M,h^*\varphi_*D')=\{s'_1,\dots,s'_{pm}\}$. Then, $i'$ is determined by the conditions $s'_i<s'_{i+1}$ and $\phi(s_i)=s'_{i'}$. Let $s:=\exp{(-f_i(M,0))}$ and $s':=\phi(s)=\exp{(-f_{i'}(\varphi_*M,0))}$. Then, (\ref{6}) implies that the number of elements of $\mathfrak{IR}(h^*\varphi_*M,h^*\varphi_*D')$ strictly greater than $s'$ are $pm-i'$, and (\ref{7}) implies that $f_{i'}(\varphi_*M,0)\ge pf_i(M,0)>-p^{-j+1}\log{\omega}$. Applying the induction hypothesis to $(\varphi_*M,\varphi_*D')\in \mathrm{Mod}(K\langle\alpha^p/t^p,t^p]\!]_0,d')$, we obtain a subobject  $(M',\varphi_*D')\subset (\varphi_*M,\varphi_*D')$ such that
\[
f_{i'+j}(\varphi_*M,0)=f_j(M',0)\text{ for }j=1,\dots,pm-i'.
\]
This implies that $(h^*M',h^*\varphi_*D')=((h^*\varphi_*M)_{>\phi(s)},h^*\varphi_*D')=((\Phi_*H^*M)_{>\phi(s)},\Phi_*H^*D')$. Let\\ $\mathcal{G}_{\varphi}(M',\varphi_*D')\subset (\varphi_*M,\varphi_*D')$ be the subobject defined in Definition~\ref{lem:G} and $(M'',D')\subset (M,D')$ the unique subobject such that $(\varphi_*M'',\varphi_*D')=\mathcal{G}_{\varphi}(M',\varphi_*D')$. By pulling-back by $h$, we have
\[
h^*(\varphi_*M'',\varphi_*D')\cong \Phi_*(H^*M'',H^*D'),
\]
\begin{align*}
h^*\mathcal{G}_{\varphi}(M',\varphi_*D')&\cong\mathcal{G}_{\Phi}(h^*M',h^*\varphi_*D')=\mathcal{G}_{\Phi}((h^*\varphi_*M)_{>\phi(s)},h^*\varphi_*D')\\
&=\mathcal{G}_{\Phi}((\Phi_*H^*M)_{>\phi(s)},\Phi_*H^*D')=\Phi_*((H^*M)_{>s},H^*D'),
\end{align*}
where the first isomorphism follows from Lemma~\ref{lem:G-pull} and the last equality follows from Lemma~\ref{lem:push}. By Lemma~\ref{lem:IR}~(ii), we have $\mathfrak{IR}(H^*M'',H^*D)=\mathfrak{IR}((H^*M)_{>s},H^*D)$ (note that $H^*(D')=(H^*D)'$). Since the cardinality of $\mathfrak{IR}((H^*M)_{>s},H^*D)$ is $m-i$ by assumption, $M''$ is of rank $m-i$ and
\[
f_{i+j}(M,0)=f_j(M'',0)\text{ for }j=1,\dots,m-i,
\]
which implies the assertion.
\section{Appendix 2: An example of rank two}
Throughout this section, assume that $p\neq 2$ and $K$ is discretely valued. In this section, we construct a differential module $M$ over $K[\![t]\!]_0$ of rank two, which corresponds to the ordinary differential equation $d^2f/dt^2-tdf/dt=0$, then we explain Main Theorem and Theorem \ref{thm:dec} by using $M$. Precisely speaking, we will explicitly describe the decomposition of $M\otimes_{K[\![t]\!]_0}K\langle\alpha/t,t]\!]_0$ for some $\alpha\in (0,1)$ given by Theorem \ref{thm:dec}. We also prove that such a decomposition does not extends to $M$. Secondly, we will prove without assuming Main Theorem that $\dim_K(M\otimes_{K[\![t]\!]_0}K\{t\})^{\nabla=0}=1$ and $M^{\nabla=0}=(M\otimes_{K[\![t]\!]_0}K\{t\})^{\nabla=0}$, while we have $(M\spcheck\otimes_{K[\![t]\!]_0}K\{t\})^{\nabla=0}=0$. We first define $M$, then prove the above results by assuming some calculations verified in the last subsection.

\begin{notation}
In addition to the notation in Appendix~1, we use the following notation.
\begin{enumerate}
\item[(1)] For $x\in\mathbb{R}$, denote by $\lfloor x\rfloor$ the maximum integer less than or equal to $x$. For $n\in\mathbb{N}_{\ge 1}$, we put $n!!=\prod_{i=0,\dots,\lfloor n/2\rfloor}(2i+1)$ if $n$ is odd, and $n!!=\prod_{i=1,\dots,\lfloor n/2\rfloor}2i$ if $n$ is even. For simplicity, we put $0!!=1$ and $(-1)!!=-1$. Note that
\[
(2n-1)!!=(2n-1)!!(2n)!!/(2n)!!=(2n)!/2^nn!=\binom{2n}{n}n!/2^n
\]
for $n\in\mathbb{N}_{\ge 1}$.
\item[(2)] For a formal power series $f=\sum_{i\in\mathbb{N}}a_it^i\in K[\![t]\!]$ with $a_i\in K$, we define the {\it radius of convergence} $R(f)\in\mathbb{R}\cup\{\infty\}$ of $f$ as $\sup\{\rho\in\mathbb{R}_{\ge 0};|a_i|\rho^i\to 0\ (i\to+\infty)\}$. By definition, $f\in K\{t\}$ if and only if $R(f)\ge 1$.
\item[(3)] For a formal sum $f=\sum_{i\in\mathbb{Z}}a_it^i$ with $a_i\in K$, we put $f'=\sum_{i\in\mathbb{Z}}ia_it^{i-1}$.
\item[(4)] Let $\lambda\in K[\![t]\!]_0$. We define the rank one differential module $V_\lambda=K[\![t]\!]_0\mathbf{e}_{\lambda}$ over $K[\![t]\!]_0$ by $D(\mathbf{e}_{\lambda})=\lambda\mathbf{e}_{\lambda}$.
\item[(5)] For an abelian group $M$ with a quotient $Q$, we denote the image of $m\in M$ in $Q$ by $m$ again if no confusion arises.
\end{enumerate}
\end{notation}

\subsection*{Definition}

Let $M=K[\![t]\!]_0e_1\oplus K[\![t]\!]_0e_2$ be the differential module over $K[\![t]\!]_0$ defined by
\[
D(e_1)=e_2,\ D(e_2)=-e_1-te_2.
\]
Then, $e_1$ is a cyclic vector of $M$ (\cite[Definition 5.4.1]{pde}), and we have the isomorphism of differential modules over $K[\![t]\!]_0$
\begin{equation}\label{eq:isom}
M\cong K[\![t]\!]_0\{T\}/K[\![t]\!]_0\{T\}(T^2+tT+1);(e_1,D(e_1))\mapsto (1,T).
\end{equation}
Note that we have $T^2+tT+1=T\cdot (T+t)$ in $K[\![t]\!]_0\{T\}$. Hence, $N=K[\![t]\!]_0(te_1+e_2)$ is a differential submodule of $M$, which is isomorphic to $V_0$. Moreover, $M/N$ is isomorphic to $V_{-t}$. Thus, we obtain an exact sequence of differential modules over $K[\![t]\!]_0$
\begin{equation}\label{eq:exact}
0\to V_0\to M\to V_{-t}\to 0.
\end{equation}

We will describe $M\spcheck$. Let $\{e_1\spcheck,e_2\spcheck\}\subset M\spcheck$ be the dual basis of $\{e_1,e_2\}\subset M$. Then, we have $D(e_1\spcheck)=e_2\spcheck$, $D(e_2\spcheck)=-e_1\spcheck+te_2\spcheck$ by definition, and $D^2(e_2\spcheck)=tD(e_2\spcheck)$. Hence, $e_2\spcheck$ is a cyclic vector of $M\spcheck$, and we have the isomorphism of differential modules over $K[\![t]\!]_0$
\begin{equation}\label{eq:isom dual}
M\spcheck\cong K[\![t]\!]_0\{T\}/K[\![t]\!]_0\{T\}(T^2-tT);(e_2\spcheck,D(e_2\spcheck))\mapsto (1,T).
\end{equation}
Thus, we may regard $M$ as the differential module corresponding to the differential equation $d^2f/dt^2-tdf/dt=0$ (see \cite[\S 5.6]{pde}).

\subsection*{Decomposition}

We have $\mathfrak{IR}(M\otimes_{K[\![t]\!]_0}\mathcal{E})=\{IR(V_{-t}\otimes_{K[\![t]\!]_0}\mathcal{E}),IR(V_0\otimes_{K[\![t]\!]_0}\mathcal{E})\}=\{\omega,1\}$ by (\ref{eq:exact}) and Example~\ref{ex:rank one}. Hence, $f_1(M,0)=-\log{\omega}>f_2(M,0)=0$. Therefore, Theorem \ref{thm:dec} is applicable to $M\otimes_{K[\![t]\!]_0}K\langle\alpha/t,t]\!]_0$ for any $\alpha\in (0,1)$. We will describe the resulting decomposition explicitly.

We identify $M$ as $K[\![t]\!]_0\{T\}/K[\![t]\!]_0\{T\}(T^2+tT+1)$ via the isomorphism (\ref{eq:isom}). Fix $\alpha\in (\omega^{1/2},1)$. We assume Lemma~\ref{lem:overconvergent} below: let $a=\sum_{i=0}^{\infty}(2i-1)!!/t^{2i+1}\in K\langle\alpha/t,t]\!]_0$. Then, $\{T+t,aT-a'\}$ is a basis of $M\otimes_{K[\![t]\!]_0}K\langle\alpha/t,t]\!]_0$ since we have
\[
\det
\begin{pmatrix}
t&-a'\\
1&a
\end{pmatrix}
=at+a'=-1.
\]
Moreover, we have
\[
D(T+t,aT-a')=(T+t,aT-a')
\begin{pmatrix}
0&0\\
0&-t
\end{pmatrix}
\]
in $M\otimes_{K[\![t]\!]_0}K\langle\alpha/t,t]\!]_0$. Hence, we obtain the decomposition of differential modules over $K\langle\alpha/t,t]\!]_0$
\[
M\otimes_{K[\![t]\!]_0}K\langle\alpha/t,t]\!]_0=K\langle\alpha/t,t]\!]_0(T+a)\oplus K\langle\alpha/t,t]\!]_0 (aT-a')
\]
with isomorphisms $K\langle\alpha/t,t]\!]_0(T+a)\cong V_0\otimes_{K[\![t]\!]_0}K\langle\alpha/t,t]\!]_0$ and $K\langle\alpha/t,t]\!]_0 (aT-a')\cong V_{-t}\otimes_{K[\![t]\!]_0}K\langle\alpha/t,t]\!]_0$. To conclude the above decomposition satisfies the condition in Theorem \ref{thm:dec}, it suffices to check $IR(V_{-t}\otimes_{K[\![t]\!]_0}F_{\rho})<IR(V_0\otimes_{K[\![t]\!]_0}F_{\rho})$ for any $\rho\in (\omega^{1/2},1)$. Since $V_0$ is trivial, we have $IR(V_0\otimes_{K[\![t]\!]_0}F_{\rho})=1$ by definition. We have $(V_{-t}\otimes_{K[\![t]\!]_0}K[\![t]\!])^{\nabla=0}=K(\mathbf{e}_{-t}\otimes\exp{(t^2/2)})$, and $R(\exp{(t^2/2)})=\omega^{1/2}$ by a similar proof as in Lemma~\ref{lem:power series}~(i). Hence, the {\it radius of convergence} $R(V_{-t})$ of $V_{-t}$ in the sense of \cite[Definition~9.3.1]{pde} is equal to $\omega^{1/2}$. We define the {\it generic radius of convergence} of $V_{-t}\otimes_{K[\![t]\!]_0}F_{\rho}$ as $R(V_{-t}\otimes_{K[\![t]\!]_0}F_{\rho})=\rho^{-1}\cdot IR(V_{-t}\otimes_{K[\![t]\!]_0}F_{\rho})$ (\cite[Definitions~9.4.4, 9.4.7]{pde}). Then, we have $R(V_{-t}\otimes_{K[\![t]\!]_0}F_{\rho})\le R(V_{-t})$ by Dwork's transfer theorem (\cite[Theorem~9.6.1]{pde}). Hence, we have $IR(V_{-t}\otimes_{K[\![t]\!]_0}F_{\rho})\le\rho\cdot R(V_{-t})=\rho\cdot\omega^{1/2}<1$, which implies the assertion.

We finally prove that $M$ is indecomposable in the category of differential modules over $K[\![t]\!]_0$. In particular, $M$ is not isomorphic to $V_0\oplus V_{-t}$. Suppose not, that is, there exists an isomorphism of differential modules over $K[\![t]\!]_0$
\[
M\cong M_1\oplus M_2
\]
such that $M_1,M_2$ are of rank one. By considering Jordan-H\"older constituents of $M$ in the category of differential modules over $K[\![t]\!]_0$, either $M_1$ or $M_2$ is isomorphic to $V_{-t}$. In particular, there exists a non-zero element $v\in M$ such that $D(v)=-tv$. Write $v=g+h T$ with $g,h\in K[\![t]\!]_0$. Then, we have $g'-h=-tg$ and $g+h'-th=-th$. Hence, $g=-h'$ and $h''+th'+h=0$. Since $R(h)\ge 1$ by $h\in K[\![t]\!]_0$, we have $h=0$ by Lemma~\ref{lem:power series}~(iii). Hence, $v=0$, which is a contradiction.

\subsection*{Horizontal sections}

We will prove $(M\otimes_{K[\![t]\!]_0}K\{t\})^{\nabla=0}=M^{\nabla=0}=K(te_1+e_2)$. By the exact sequence (\ref{eq:exact}), we obtain an exact sequence
\[
0\to (V_0\otimes_{K[\![t]\!]_0}K\{t\})^{\nabla=0}\to (M\otimes_{K[\![t]\!]_0}K\{t\})^{\nabla=0}\to (V_{-t}\otimes_{K[\![t]\!]_0}K\{t\})^{\nabla=0}.
\]
Since $(V_{-t}\otimes_{K[\![t]\!]_0}K[\![t]\!])^{\nabla=0}=K(\mathbf{e}_{-t}\otimes\exp{(t^2/2)})$ and $R(\exp{(t^2/2)})=\omega^{1/2}<1$, we have $(V_{-t}\otimes_{K[\![t]\!]_0}K\{t\})^{\nabla=0}=0$. Hence, we obtain a canonical isomorphism $(M\otimes_{K[\![t]\!]_0}K\{t\})^{\nabla=0}\cong (V_0\otimes_{K[\![t]\!]_0}K\{t\})^{\nabla=0}$. Obviously, we have $(V_0\otimes_{K[\![t]\!]_0}K\{t\})^{\nabla=0}=K(\mathbf{e}_0\otimes 1)$, which implies the assertion.

We will prove $(M\spcheck\otimes_{K[\![t]\!]_0}K\{t\})^{\nabla=0}=0$ by assuming Lemma~\ref{lem:power series}. We identify $M\spcheck$ as \\ $K[\![t]\!]_0\{T\}/K[\![t]\!]_0\{T\}(T^2-tT)$ via the isomorphism (\ref{eq:isom dual}). We claim that the $K$-vector space $(M\spcheck\otimes_{K[\![t]\!]_0}K[\![t]\!])^{\nabla=0}$ admits the basis $\{bT,-1+cT\}$, where
\[
b=\exp{(-t^2/2)},\ c=\sum_{i=0}^{\infty}((-1)^{i}/(2i+1)!!)t^{2i+1}.
\]
By Lemma~\ref{lem:power series}~(iii), $bT,-1+cT\in M\spcheck$ are linearly independent over $K$. If $g+hT\in M\spcheck\otimes_{K[\![t]\!]_0}K[\![t]\!]$ with $g,h\in K[\![t]\!]$ satisfies $D(g+hT)=0$, then we have $g'=0$ and $g+th+h'=0$. Hence, we have $g\in K$, and, by differentiating the second equation, we have $h''+th'+h=0$. Therefore, we have $h\in Kb+Kc$ by Lemma~\ref{lem:power series}~(iii). If we write $h=d_1b+d_2c$, then we have $g=-(th+h')=-d_2$ by Lemma~\ref{lem:power series}~(ii), which implies the claim. Let $v=g+hT\in (M\spcheck\otimes_{K[\![t]\!]_0}K\{t\})^{\nabla=0}$ with $g,h\in K\{t\}$. Then, $R(h)\ge 1$ by $h\in K[\![t]\!]_0$. By the above claim, we can write $v=d_1bT+d_2(-1+cT)$ for some $d_1,d_2\in K$. Then, we have $h=d_1 b+d_2 c$. If $h\neq 0$, then $R(h)=\omega^{1/2}<1$ by Lemma~\ref{lem:power series}~(iii), which is a contradiction. Hence, $h=0$. By Lemma~\ref{lem:power series}~(iii), $d_1=d_2=0$, which implies the assertion.

\subsection*{Some power series}\label{subsec:calculation}

\begin{lem}\label{lem:factorial}
For any integer $i\ge 1$, we have
\[
i/(p-1)-(1+\log_p{i})\le v_p(i!)\le i/(p-1).
\]
In other words, $\omega^i\le |i!|\le i\omega^{i-(p-1)}$.
\end{lem}
\begin{proof}
We have $v_p(i!)=\lfloor i/p^k\rfloor\le\sum_{k=1}^{\infty}(i/p^k)=i/(p-1)$. We choose $r\in\mathbb{N}$ such that $p^r\le i<p^{r+1}$. It suffices to prove $i/(p-1)-v_p(i!)\le r+1$. We obtain the assertion by
\begin{align*}
&i/(p-1)-v_p(i!)=\sum_{k=1}^r\{(i/p^k)-\lfloor i/p^k\rfloor\}+\sum_{k=r+1}^{\infty}i/p^k\le \sum_{k=1}^r((p-1)+\dots+(p-1)p^{k-1})/p^k+\sum_{k=r+1}^{\infty}i/p^k\\
=&r-(p^r-1)/p^r(p-1)+i/p^r(p-1)\le r-(p^r-1)/p^r(p-1)+(p^{r+1}-1)/p^r(p-1)=r+1.
\end{align*}
\end{proof}

\begin{lem}\label{lem:overconvergent}
We consider the following formal sum with coefficients in $\mathbb{Z}_p$
\[
a=\sum_{i=0}^{\infty}(2i-1)!!t^{-(2i+1)}=-1/t+1/t^3+\dots.
\]
Then, for any $\alpha\in (\omega^{1/2},1)$, we have $a\in K\langle\alpha/t,t]\!]_0$. Moreover, we have
\[
a'+ta+1=0,\ a''+ta'+a=0,
\]
and, in $K\langle\alpha/t,t]\!]_0\{T\}$,
\[
a(T^2+tT+1)=(T+t)(aT-a').
\]
\end{lem}
\begin{proof}
Since we have
\[
|(2i-1)!!|\alpha^{-(2i+1)}=\left|\binom{2i}{i}i!/2^i\right|\alpha^{-(2i+1)}\le |i!|\alpha^{-(2i+1)}\le i(\omega/\alpha^2)^i\omega^{-(p-1)}\alpha^{-1}
\]
by Lemma~\ref{lem:factorial}, we have $|(2i-1)!!|\alpha^{-(2i+1)}\to 0$ as $i\to+\infty$. Hence, $a\in K\langle\alpha/t,t]\!]_0$. The rest of the assertion follows by a direct calculation.
\end{proof}

\begin{lem}\label{lem:power series}
We consider the following formal power series over $\mathbb{Q}_p$
\[
b=\exp{(-t^2/2)}=\sum_{i=0}^{\infty}((-1)^i/2^ii!)t^{2i}=1-t^2/2+\dots,
\]
\[
c=\sum_{i=0}^{\infty}((-1)^i/(2i+1)!!)t^{2i+1}=t-t^3/3+\dots.
\]
\begin{enumerate}
\item[(i)] We have $R(b)=R(c)=\omega^{1/2}$.
\item[(ii)] We have $b'+tb=0$ and $c'+tc=1$.
\item[(iii)] We consider the $K$-vector space
\[
W=\{h\in K[\![t]\!];h''+th'+h=0\}.
\]
Then, $b,c\in W$, and, $\{b,c\}$ forms a basis of $W$. Moreover, for any non-zero $h\in W$, we have $R(h)=\omega^{1/2}$.
\end{enumerate}
\end{lem}
\begin{proof}
\begin{enumerate}
\item[(i)] Let $\rho\in (0,\omega^{1/2})$. By Lemma \ref{lem:factorial}, we have $i^{-1}\omega^{-i+(p-1)}\le |(-1)^i/2^ii!|\le \omega^{-i}$. Since for $\rho\in (0,\omega^{1/2})$, we have $\omega^{-i}\rho^{2i}\to 0$ as $i\to+\infty$, and, for $\rho\in (\omega^{1/2},\infty)$, we have $i^{-1}\omega^{-i+(p-1)}\rho^{2i}\to\infty$ as $i\to+\infty$, we obtain $R(b)=\omega^{1/2}$. We have $1/(2i+1)!!=2^ii!/(2i+1)!$, and,
\[
|1/(2i+1)!!|\le i\omega^{i-(p-1)}/\omega^{2i+1}=i\omega^{-i-p},
\]
\[
|1/(2i+1)!!|\ge \omega^i/((2i+1)\omega^{2i+1-(p-1)})=\omega^{-i+p-2}/(2i+1)
\]
by Lemma \ref{lem:factorial}. Similarly as in the case of $b$, we obtain $R(c)=\omega^{1/2}$.
\item[(ii)] It follows from a direct calculation.
\item[(iii)] By differentiating the equations in (ii), we have $b,c\in W$. The rest of the assertion follows by noting that $\dim_KW=2$, and, all coefficients of odd (resp. even) powers of $t$ in $b$ (resp. $c$) are $0$.
\end{enumerate}
\end{proof}

\section*{Acknowledgement}
The author thanks the referee for useful comments. This work is supported by JSPS KAKENHI Grant-in-Aid for Young Scientists (B) JP17K14161.


\begin{thebibliography}{}
\bibitem[CT09]{CT}
B.~Chiarellotto and N.~Tsuzuki, Logarithmic growth and Frobenius filtrations for solutions of $p$-adic differential equations, J. Inst. Math. Jussieu 8 (2009),  no. 3, 465--505.
\bibitem[CT11]{CT2}
B.~Chiarellotto and N.~Tsuzuki, Log-growth filtration and Frobenius slope filtration of $F$-isocrystals at the generic and special points, Doc. Math. 16 (2011), 33--69.
\bibitem[Chr83]{Chr}
G.~Christol, Modules diff\'erentiels et \'equations diff\'erentielles $p$-adiques, Queen's Papers in Pure and Applied Mathematics, 66. Queen's University, Kingston, ON,  1983. vi+218 pp.
\bibitem[Dwo73]{Dw}
B.~Dwork,  On $p$-adic differential equations. II. The $p$-adic asymptotic behavior of solutions of ordinary linear differential equations with rational function coefficients, Ann. of Math. (2)  98  (1973), 366--376.
\bibitem[Ked09]{ss}
K.~Kedlaya, Semistable reduction for overconvergent $F$-isocrystals, III: Local semistable reduction at monomial valuations, Compos. Math. 145 (2009), no.1, 143--172.
\bibitem[Ked10]{pde}
K.~Kedlaya, $p$-adic differential equations, Cambridge Studies in Advanced Mathematics, 125. Cambridge University Press, Cambridge,  2010. xviii+380 pp.
\bibitem[KX10]{KX}
K.~Kedlaya-L.~Xiao, Differential modules on $p$-adic polyannuli. J. Inst. Math. Jussieu 9 (2010), no.1, 155--201.
\bibitem[Ohk17]{Ohk}
S.~Ohkubo, On the rationality and continuity of logarithmic growth filtration of solutions of $p$-adic differential equations, Adv. Math. 308 (2017), 83--120.
\bibitem[Rob80]{Rob}
P.~Robba, Lemmes de Hensel pour les operateurs diff\'erentiels. Application \`a la r\'eduction formelle des \'equations diff\'erentielles, Enseign. Math.(2) 26 (1980), no.3--4, 279--311.
\end{thebibliography}
\end{document}